\documentclass[reqno]{amsproc}
\usepackage[utf8]{inputenc}
\usepackage{cite}
\usepackage{xcolor}
\usepackage{color}
\usepackage{calligra}
\usepackage[T1]{fontenc}
\usepackage{latexsym}
\usepackage[all]{xy}
\usepackage{hyperref}
\usepackage{mathrsfs}
\usepackage{amsmath,amsfonts,amssymb,amsxtra}
\numberwithin{equation}{section}
\newtheorem{theorem}{\bf Theorem}[section]

\newtheorem{lem}{\bf Lemma}[section]
\newtheorem{cor}{\bf Corollary}[section]
\newtheorem{remark}{\bf Remark}[section]

\newcommand\dv{\mathrm{div}}
\newcommand\vol{\mathrm{vol}}
\newcommand\tr{\mathrm{tr}}
\newcommand\rc{\mathrm{Ric}}
\usepackage{algorithmic}
\begin{document}

\title[Eigenvalue estimates for elliptic operators in divergence form]{Eigenvalue estimates for a class of elliptic differential operators in divergence form on Riemannian manifolds isometrically immersed in Euclidean space}
\author[M.C. Araújo Filho]{Marcio C. Araújo Filho$^1$}
\author[J.N.V. Gomes]{José N.V. Gomes$^2$}
\address{$^1$Departamento de Matemática, Universidade Federal de Rondônia, Campus Ji-Paraná, R. Rio Amazonas, 351, Jardim dos Migrantes, 76900-726 Ji-Paraná, Rondônia, Brazil}
\address{$^2$Departamento de Matemática, Universidade Federal de São Carlos, Rod. Washington Luíz, Km 235, 13565-905 São Carlos, São Paulo, Brazil}
\email{$^1$marcio.araujo@unir.br}
\email{$^2$jnvgomes@ufscar.br}
\keywords{Eigenvalue estimates, Elliptic operator, Riemannian Manifold, Gaussian soliton.}
\subjclass[2010]{Primary 47A75; Secondary 47F05, 35P15, 53C25}

\begin{abstract}
We obtain eigenvalue estimates for a larger class of elliptic differential operators in divergence form on a bounded domain in a complete Riemannian manifold isometrically immersed in Euclidean space. As an application, we give eigenvalue estimates in the Gaussian shrinking soliton, and we find a domain that makes the behavior of these estimates similar to the estimates for the case of the Laplacian. Moreover, we also give an answer to the generalized conjecture of Pólya.
\end{abstract}
\maketitle

\section{Introduction}\label{Introduction}
Let $(M^n,\langle,\rangle)$ be an $n$-dimensional complete Riemannian manifold, and $\Omega~\subset~M^n$ be a bounded domain with smooth boundary $\partial\Omega$. Let us consider a function $\eta\in C^2(M)$ and a symmetric positive definite $(1,1)$-tensor $T$ on $M^n$. Since $\Omega$ is a bounded domain, there exist two positive real constants $\varepsilon$ and $\delta$, such that $\varepsilon \leq\langle T(X), X \rangle\leq\delta$, for any unit vector field $X$ on $\Omega$.

In this paper, we compute eigenvalue estimates for a larger class of elliptic differential operators in divergence form that includes, e.g., the Laplace-Beltrami and Cheng-Yau operators, to name a few. We consider the eigenvalue problem with the Dirichlet boundary condition:
\begin{equation}\label{problem1}
    \left\{\begin{array}{ccccc}
    \mathscr{L}  u  &=& - \lambda u & \mbox{in } & \Omega,\\
     u&=&0 & \mbox{on} & \partial\Omega,
    \end{array}
    \right.
\end{equation}
where $\mathscr{L}$ is defined as the second-order elliptic differential operator in the $(\eta,T)$-divergence form
\begin{equation}\label{eq1.1}
    \mathscr{L}f =\dv_\eta (T(\nabla f)) = \dv(T(\nabla f)) - \langle \nabla \eta, T(\nabla f) \rangle.
\end{equation}
Here $\dv$ stands for the divergence of smooth vector fields and $\nabla$ for the gradient of smooth functions.

We observe that $\mathscr{L}$ is a formally self-adjoint operator in the Hilbert space $\mathcal{H}_0^1(\Omega,e^{-\eta}d\Omega)$ of all functions in $L^2(\Omega,e^{-\eta}d\Omega)$ that vanish on $\partial \Omega$ in the sense of the trace, see Section~\ref{preliminaries}. Thus, Problem~\eqref{problem1} has a real and discrete spectrum
\begin{equation}\label{spectrum}
    0 < \lambda_1 \leq \lambda_2 \leq \cdots \leq \lambda_k \leq \cdots\to\infty,
\end{equation}
where each $\lambda_i$ is repeated according to its multiplicity. Eigenspaces belonging to distinct eigenvalues are orthogonal in $L^2(\Omega,e^{-\eta}d\Omega)$, which is the direct sum of all the eigenspaces. We refer to the dimension of each eigenspace as the multiplicity of the eigenvalue, for more details see Chavel~\cite{chavel}. 

For the case of $T$ be the identity operator, and $\eta$ be a constant function, Problem~\eqref{problem1} becomes
\begin{equation}\label{problem-3.4}
    \left\{\begin{array}{ccccc}
    \Delta  u  &=& - \lambda u & \mbox{in } & \Omega,\\
     u&=&0 & \mbox{on} & \partial\Omega,
    \end{array}
    \right.
\end{equation}
 where $\Delta$ is the Laplace-Beltrami operator. Some interesting eigenvalue estimates for Problem~\eqref{problem-3.4} had been obtained, for instance, see Chen and Cheng~\cite{ChenCheng08}, Cheng and Yang~\cite{ChengYang07,ChengYang05,ChengYang06} and Yang~\cite{Yang}, some of them will be discussed in due course.
 
For the case of $T$ be the identity operator and $\eta$ be not necessarily a constant function, Problem~\eqref{problem1} becomes
\begin{equation}\label{problem-3.5}
    \left\{\begin{array}{ccccc}
    \Delta_\eta  u  &=& - \lambda u & \mbox{in } & \Omega,\\
     u&=&0 & \mbox{on} & \partial\Omega,
    \end{array}
    \right.
\end{equation}
where $\Delta_\eta$ is the drifted Laplace-Beltrami operator. In this case, it is worth mentioning the work by the second author and Miranda~\cite{GomesMiranda}, and the paper by Xia and Xu~\cite{XiaXu} in which we can find some estimates of the eigenvalues of Problem~\eqref{problem-3.5}. Later, we will undertake a more detailed discussion of eigenvalues estimates of Problem~\eqref{problem-3.5} proved in these papers. 

Problem~\eqref{problem1} is a partial differential equation (PDE) with the Dirichlet boundary condition. It is well-known that PDE's play a fundamental role not only from a mathematical point of view but also in the description and modeling of many physical and probabilistic phenomena. Such equations appear, for example, in Laplace's equations, Helmholtz's equation, linear transport equation, Liouville's equation, Kolmogorov's equation, Schrödinger's equation, and in a differential geometry context. An interesting example appears in the equation of minimal surfaces, see, e.g., Evans~\cite{Evans} or Grebenkov and Naguyen~\cite{GrebenkorNaguyen}. In particular, Schrödinger's equation is a central equation in quantum mechanics. For instance, the eigenvalues of Schrödinger's equation corresponding to the allowed energy levels of the quantum system, and the gap between the eigenvalues is the gap between the energy levels. These eigenvalues are related  to the Hamiltonian operator that appears in Schrödinger's equation. Indeed, this equation is an eigenvalue problem for the Hamiltonian operator where the eigenvalues are the (allowed) total energies.

In what follows, the smooth vector field ${\bf H}_T$ on $M^n$ stands for the generalized mean curvature vector  associated with the $(1,1)$-tensor $T$, see Section~\ref{preliminaries} for details. Our first result is a quadratic estimate for the eigenvalues of Problem~\eqref{problem1}, which is an essential tool to obtain some of our estimates.

\begin{theorem}\label{theorem1.1}
Let $\Omega$ be a bounded domain in an $n$-dimensional complete Riemannian manifold $M^n$ isometrically immersed in $\mathbb{R}^m$, and $\lambda_i$ be the $i$-th eigenvalue of Problem~\eqref{problem1}. Then, we have
\begin{equation*}
    \sum_{i=1}^k(\lambda_{k+1}-\lambda_i)^2 \leq \frac{4\delta}{n\varepsilon}\sum_{i=1}^k(\lambda_{k+1}-\lambda_i)\Big[\Big(\sqrt{\lambda_i}  + \frac{T_0}{2\sqrt{\delta}}\Big)^2 + \frac{n^2H_0^2+4C_0+2\delta T_0\eta_0}{4\delta} \Big],
\end{equation*}
where
\begin{equation}\label{C_0}
    C_0=\sup_\Omega \Big\{\frac{1}{2}\dv (T^2(\nabla \eta)) - \frac{1}{4}|T(\nabla \eta)|^2\Big\},
\end{equation}
$T_0=\sup_{\Omega}|\tr(\nabla T)|$, $\eta_0=\sup_{\Omega}|\nabla \eta|$ and $H_0=\sup_\Omega |{\bf H}_T|$.
\end{theorem}

We also prove an estimate for the sum of lower order eigenvalues in terms of the first eigenvalue of Problem~\eqref{problem1}.

\begin{theorem}\label{theorem1.2}
Let $\Omega$ be a bounded domain in an $n$-dimensional complete Riemannian manifold $M^n$ isometrically immersed in $\mathbb{R}^m$, and denote by $\lambda_i$ the $i$-th eigenvalue of Problem~\eqref{problem1}, for $i=1,\ldots,n$. Then, we have
\begin{equation*}
    \sum_{i=1}^n(\lambda_{i+1} - \lambda_1) \leq \frac{4\delta}{\varepsilon}\Big[\Big(\sqrt{\lambda_1} +\frac{T_0}{2\sqrt{\delta}}\Big)^2 + \frac{n^2H_0^2+4C_0+2\delta T_0\eta_0}{4\delta}\Big],
\end{equation*}
 where the constants  $T_0, H_0, C_0$ and $\eta_0$ are as in Theorem~\ref{theorem1.1}.
\end{theorem}

It is important to observe that there exists a beautiful method of deducing inequalities similar to Hile and Protter~\cite{HileProtter} and Payne et al.~\cite{Payne-etal}. Moreover, it also recover Inequality~(4) in Yang~\cite{Yang} which is sharper than corresponding inequalities in\cite{HileProtter} and \cite{Payne-etal}, in the Weyl's asymptotic sense, see comments in \cite[p. 2-3]{Yang}. Such a method is abstract and can be applied to a given self-adjoint operator $\mathscr{L}$ acting in a Hilbert space using purely functional-analytic techniques, see Harrell II and Stubbe~\cite{HarrellStubbe} and Levitin and Parnovski~\cite{LevitinParnovski} for details. This abstract method consists in computing the commutators of $\mathscr{L}$ with some auxiliary operator $G$: the bounds are then expressed in terms of action of the commutators on eigenfunctions of $\mathscr{L}$. In each particular case, a suitable choice of $G$ is required (for a second order operator the coordinate functions usually serve well),
and then some specific tricks involving integration by parts and Cauchy–Schwartz inequality. Unfortunately, we were not yet able to use this method to get the quadratic inequality in Theorem~\ref{theorem1.1}. For a general case of Weyl's asymptotic formula, see e.g. Fonseca and Gomes~\cite[Expressions~(7.8) and (7.9)]{FonsecaGomes}.


\begin{remark}
Theorems~\ref{theorem1.1} and \ref{theorem1.2} generalize two results by Chen and Cheng \cite[Theorem~1.1]{ChenCheng08} and ~\cite[Theorem~1.2]{ChenCheng08}, respectively. This will become more evident by means of Corollaries~\ref{cor-Tparalell} and \ref{lower-order-estimate}.
\end{remark}

A special case that will be addressed in this paper refers to the {\it divergence-free} tensors, i.e., $\dv{T} = 0$. In this setting, Theorems~\ref{theorem1.1} and~\ref{theorem1.2} have a simpler configuration, see Corollaries~\ref{cor-Tparalell} and~\ref{lower-order-estimate}, respectively. Divergence-free tensors often arise from physical facts, and we can find some of them in fluid dynamics such as compressible gas, rarefied gas, steady/self-similar flows, and relativistic gas dynamics, for more details see Serre~\cite{Serre}.

When $T$ is divergence-free there exists a relationship between the operator $\mathscr{L}$ and the operator which was introduced by Cheng and Yau~\cite{ChengYau} as follows 
\begin{equation*}
    \square f = \tr{(\nabla^2f \circ T)}=\langle \nabla^2 f, T\rangle,
\end{equation*}
where $f\in C^\infty(M)$. Indeed, from \cite[Eq.~(2.3)]{GomesMiranda} the operator $\mathscr{L}$ becomes
\begin{equation}\label{drifted-CY-operator}
    \mathscr{L}f = \square f - \langle\nabla \eta, T(\nabla f)\rangle.
\end{equation}
Thus, $\mathscr{L}$ is a first-order perturbation of the  Cheng-Yau operator $\square$.  We call Eq.~\eqref{drifted-CY-operator} a \emph{drifted Cheng-Yau operator} with a \emph{drifting function} $\eta$. In particular, if $\eta$ is constant, then $\square$ is a Cheng-Yau operator with $\dv T=0.$

Cheng-Yau operator is appropriate for studying complete hypersurfaces of constant scalar curvature in space forms. In fact, with a careful study of this operator, Cheng and Yau obtained interesting rigidity results for such hypersurfaces. For instance, in the case of Euclidean space ambient, they proved that the only complete and non-compact hypersurfaces with non-negative constant normalized scalar curvature and non-negative sectional curvature are the generalized cylinders, see~\cite{ChengYau} for details. 

One special case of divergence-free symmetric tensor on an $n$-dimensional Riemannian manifold $(M^n, \langle , \rangle)$ is the Einstein tensor $E = \rc - \frac{R}{2}\langle , \rangle$, where $\rc$ is the Ricci tensor of the metric $\langle , \rangle$ and $R=\tr(\rc)$, we can see this fact from contracted Bianchi identity. Another case is when $(M^n, \langle , \rangle)$, $n\geq 3$, is an Einstein manifold, that is, $\rc=\frac{R}{n} \langle , \rangle$, then, $\rc$ is a tensor symmetric positive definite (if $R>0$) which is divergence-free, since $R$ is constant by Schur's lemma and $\dv Ric=\frac{dR}{n}$. The tensors $T = - \rc$ and $-S$ on the hyperbolic space $(\mathbb{H}^n(-1), \langle , \rangle)$, where $S=\frac{1}{n-2}\Big(\rc - \frac{R}{2(n-1)}\langle , \rangle\Big)$ is the Schouten tensor of $\langle , \rangle$, for $n\geq 3$, are two divergence-free positive definite symmetric tensors.

If $T$ is divergence-free, then $T_0=\sup_{\Omega}|\tr(\nabla T)|=0$ (see Section~\ref{preliminaries}), so that from Theorem~\ref{theorem1.1} we obtain immediately the following eigenvalue estimate for the drifted Cheng-Yau operator.
\begin{cor}\label{cor-Tparalell}
Let $\Omega$ be a bounded domain in an $n$-dimensional complete Riemannian manifold $M^n$ isometrically immersed in $\mathbb{R}^m$, and $\lambda_i$ be the $i$-th eigenvalue of Problem~\eqref{problem1} for the drifted Cheng-Yau operator. Then, we have
\begin{equation*}
    \sum_{i=1}^k(\lambda_{k+1}-\lambda_i)^2 \leq \frac{4\delta}{n\varepsilon}\sum_{i=1}^k(\lambda_{k+1}-\lambda_i)\left(\lambda_i+\frac{n^2H_0^2+4C_0}{4\delta}\right),
\end{equation*}
where $C_0$ is given by \eqref{C_0} and $H_0=\sup_\Omega |{\bf H}_T|$.
\end{cor}

For applications of Corollary~\ref{cor-Tparalell}, we set
\begin{equation}\label{sigma-i}
    \varsigma_i=\lambda_i+\frac{n^2H_0^2+4C_0}{4\delta}.
\end{equation}

With this new notation, inequality in Corollary~\ref{cor-Tparalell} is equivalent to
\begin{equation}\label{T-parallel}
    \sum_{i=1}^k(\varsigma_{k+1}-\varsigma_i)^2 \leq \frac{4\delta}{n\varepsilon}\sum_{i=1}^k(\varsigma_{k+1}-\varsigma_i)\varsigma_i.
\end{equation}
Moreover, from \eqref{spectrum} we have $0<\varsigma_1 \leq \varsigma_2 \leq \cdots \leq \varsigma_{k}\leq \cdots\to\infty$. So, we can apply the recursion formula of Cheng and Yang~\cite{ChengYang07}, from which we immediately obtain the next result.

\begin{cor}
Under the same setup as in Corollary~\ref{cor-Tparalell}, we have
\begin{equation*}
    \varsigma_{k+1} \leq \Big(1+\frac{4\delta}{n\varepsilon}\Big)k^{\frac{2\delta}{n\varepsilon}}\varsigma_1.
\end{equation*}
\end{cor}
For the Laplace-Beltrami operator, we have $\varepsilon=\delta=1$, then from the classical Weyl’s asymptotic formula for the eigenvalues~\cite{Weyl}, we know that the previous estimate is optimal in the sense of the order on $k$. 

Let us summarize more applications of Corollary~\ref{cor-Tparalell}. We begin with a second Yang type inequality, namely,
\begin{equation}\label{Equation-1.6}
    \varsigma_{k+1} \leq \Big(1+\frac{4\delta}{n\varepsilon}\Big)\frac{1}{k}\sum_{i=1}^k\varsigma_i,
\end{equation}
which follows by a direct computation from \eqref{T-parallel}. It generalizes the second Yang inequality, see Ashbaugh~\cite[Inequality~(1.10)]{Ashbaugh}.

Inequality~\eqref{T-parallel} is a quadratic inequality  of $\lambda_{k+1}$, which generalizes the first Yang inequality (see~\cite[Inequality~(1.7)]{Ashbaugh}), and by solving it we obtain an upper bound for $\lambda_{k+1}$, and also an estimate for the gap between consecutive eigenvalues, see Corollary~\ref{cor-yang-type-ineq}.

As already noted in \cite{Ashbaugh}, one can be shown that the first Yang inequality implies the second Yang inequality, which in turn implies the inequality of Hile and Protter~\cite{HileProtter}(see Inequality~1.6 in \cite{Ashbaugh}), besides, the inequality of Hile and Protter implies the inequality of Payne et al.~\cite{Payne-etal}(see Inequality~1.5 in \cite{Ashbaugh}). In the same way, we highlight that the first inequality in Corollary~\ref{cor-yang-type-ineq} is better than Inequality~\ref{Equation-1.6}.

\begin{cor}\label{cor-yang-type-ineq}
Under the same setup as in Corollary~\ref{cor-Tparalell}, we have
\begin{align*}
    \varsigma_{k+1} \leq& \Big(1+\frac{2\delta}{n\varepsilon}\Big)\frac{1}{k}\sum_{i=1}^k\varsigma_i+\Big[\Big(\frac{2\delta}{n\varepsilon}\frac{1}{k}\sum_{i=1}^k\varsigma_i\Big)^2 - \Big(1+\frac{4\delta}{n\varepsilon}\Big)\frac{1}{k}\sum_{j=1}^k\Big(\varsigma_j -  \frac{1}{k}\sum_{i=1}^k\varsigma_i\Big)^2 \Big]^{\frac{1}{2}}
\end{align*}
and
\begin{equation*}
    \varsigma_{k+1} - \varsigma_k\leq 2\Big[\Big(\frac{2\delta}{n\varepsilon}\frac{1}{k}\sum_{i=1}^k\varsigma_i\Big)^2 - \Big(1+\frac{4\delta}{n\varepsilon}\Big)\frac{1}{k}\sum_{j=1}^k\Big(\varsigma_j -  \frac{1}{k}\sum_{i=1}^k\varsigma_i\Big)^2 \Big]^{\frac{1}{2}}.
\end{equation*}
\end{cor}

In the Laplacian case, Inequality~\eqref{Equation-1.6} is just the second Yang type inequality by Chen and Cheng~\cite[Inequality~(1.8)]{ChenCheng08}, whereas the first and second inequalities in Corollary~\ref{cor-yang-type-ineq} generalize Theorem~1 and Corollary~1 in Cheng and Yang~\cite{ChengYang05}, respectively, as well as Theorem~1 and Corollary~1 in Cheng and Yang~\cite{ChengYang06}.
 
From the Conjecture of Pólya~\cite{polya1961} and the work of Li and Yau~\cite{liYau1983}, Cheng and Yang~\cite{ChengYang09} proposed on a bounded domain $\Omega$ in an $n$-dimensional complete Riemannian manifold $M$ {\it the generalized conjecture of Pólya} and they gave a partial answer for it, see ~\cite[Theorem 1.1]{ChengYang09}. More recently, the second author and Miranda obtained an extension of the result of Cheng and Yang for the drifted Laplacian operator, see~\cite[Theorem~2]{GomesMiranda}. Here, we improve this latter result (see 
Remark~\ref{Impovement-CPolya}) as follows:
\begin{cor}\label{cor-conjectura}
Let $\Omega$ be a bounded domain in an $n$-dimensional complete Riemannian manifold $M^n$ isometrically immersed in $\mathbb{R}^m$, and $\lambda_i$ be the $i$-th eigenvalue of Problem~\eqref{problem1} for the drifted Laplacian operator. Then, we have
\begin{equation*}
    \frac{1}{k}\sum_{i=1}^k\varsigma_i \geq \frac{n}{\sqrt{(n+2)(n+4)}}\frac{4\pi^2}{(\omega_n \vol(\Omega) )^{\frac{2}{n}}}k^{\frac{2}{n}}, 
\end{equation*}
where $\varsigma_i = \lambda_i + \frac{n^2H_0^2+4C_0}{4}$, with $C_0 = \sup_\Omega \{\frac{1}{2}\Delta \eta - \frac{1}{4}|\nabla \eta|^2\}$ and $H_0=\sup_\Omega |{\bf H}|$.
\end{cor}
 
Theorem~\ref{theorem1.2} immediately implies the next result.
\begin{cor}\label{lower-order-estimate}
Let $\Omega$ be a bounded domain in an $n$-dimensional complete Riemannian manifold $M^n$ isometrically immersed in $\mathbb{R}^m$ and denote by $\lambda_i$ the $i$-th eigenvalue of Problem~\eqref{problem1} for the drifted Cheng-Yau operator, for $i~=~1,\ldots,n$. Then, we have
\begin{equation*}
    \frac{\varsigma_2+\cdots+\varsigma_{n+1}}{\varsigma_1}\leq\frac{4\delta}{\varepsilon}+n.
\end{equation*}
\end{cor}

If $\Omega$ is a bounded domain in an $n$-dimensional complete Riemannian manifold $M^n$ isometrically minimally immersed in $\mathbb{R}^m$, $\eta = constant$ and $T=I$, then ${\bf H}_T={\bf H}=0$ and from Corollaries~\ref{cor-Tparalell} and \ref{lower-order-estimate}, we have
\begin{equation}\label{Equation-(1.8)}
    \sum_{i=1}^k(\lambda_{k+1}-\lambda_i)^2 \leq \frac{4}{n}\sum_{i=1}^k(\lambda_{k+1}-\lambda_i)\lambda_i \quad \mbox{and} \quad \frac{\lambda_2 + \cdots + \lambda_{n+1}}{\lambda_1} \leq 4+n,
\end{equation}
which are the inequalities obtained by Chen and Cheng \cite{ChenCheng08} in Corollaries~1.2 and 1.3, respectively, and they observed that once $\mathbb{R}^n$ is minimally immersed in $\mathbb{R}^m, n<m$, then the first inequality in \eqref{Equation-(1.8)} recovers the first Yang's inequality~\cite{Yang}. Therefore, as we already before mentioned, the quadratic inequality in Theorem~\ref{theorem1.1} is an extension of the quadratic inequality of Yang~\cite{Yang} to the operator $\mathscr{L}$ in $(\eta, T)$-divergence form in bounded domains in a Riemannian manifold isometrically immersed in a Euclidean space. 
 
It is worth mentioning here an interesting geometric interpretation of the constant $C_0$ in Corollary~\ref{cor-conjectura} as before mentioned by the authors of the present paper in \cite{GomesAraujoFilho}. If $g = g_0 + e^{-\eta}d\theta^2$ is the warped metric on the product $\Omega\times \mathbb{S}^1$, where $g_0$ stands for the canonical metric on the domain $\Omega\subset\mathbb{R}^n$, whereas $d\theta^2$ is the canonical metric of the unit sphere $\mathbb{S}^1$, then the scalar curvature of $g$ is given by $\frac{1}{2}\Delta \eta - \frac{1}{4}|\nabla \eta |^2$.
Hence, $C_0$ can be obtained as the supremum of the scalar curvature with respect to the warped metric $g$ on $\Omega \times \mathbb{S}^1$. Furthermore, notice that the constant $C_0$ appears naturally since it depends only on the drifting function $\eta$ and we do not place any conditions on this function. Here, we find a domain and a drifting function $\eta$ to answer positively the following natural question:

\vspace{0.2cm}
\emph{Is it possible to get some domain so that the inequalities obtained from Theorems~\ref{theorem1.1} and \ref{theorem1.2} do not depend on the constant $C_0$ for some non-trivial drifting function $\eta$?}

\vspace{0.2cm}
We give a positive answer to this question by applying our results for the Gaussian shrinking soliton, see Remark~\ref{remark-1.2}. Recall that the triple $(M, \langle,\rangle, \eta)$ is called a gradient Ricci soliton
if the Bakry-Emery Ricci tensor $\rc_\eta := \rc + \nabla^2 \eta$ is a multiple of its metric $\langle,\rangle$, i.e. $\rc_\eta=\lambda \langle,\rangle$, for some constant $\lambda$. They are self-similar solutions of the Hamilton-Ricci flow and are classified according to the sign of $\lambda$: It is called steady for $\lambda = 0$, shrinking for $\lambda>0$, and expanding for $\lambda<0$. In particular, the Gaussian shrinking soliton is the triple $(\mathbb{R}^n,\langle,\rangle_{can}, \frac{1}{4}|x|^2)$, where $\langle,\rangle_{can}$ is the standard Euclidean metric on $\mathbb{R}^n$. We highlight that the Bakry-Emery Ricci tensor $\rc_\eta$ on the gradient shrinking Ricci soliton is another example of a symmetric and positive definite tensor, which is also divergence-free.

In the case of the Gaussian shrinking soliton, we take $\eta(x) = \frac{1}{4}|x|^2$, and then $\Delta \eta = \frac{n}{2}$ and $|\nabla \eta|^2=\frac{1}{4}|x|^2$, hence, for $T=I$, from \eqref{C_0} we get
\begin{equation*}
    C_0 = \sup_\Omega \Big\{\frac{1}{2}\Delta \eta - \frac{1}{4}|\nabla \eta|^2\Big\} \leq \frac{n}{4}-\frac{1}{16}\inf_{{\Omega}}|x|^2.
\end{equation*}

From the previous inequality and Theorems~\ref{theorem1.1} and \ref{theorem1.2}, we obtain the following inequalities of eigenvalues of the Dirichlet problem for the drifted Laplacian operator on the Gaussian shrinking soliton.
\begin{cor}\label{gaussian-corollary}
Let $\Omega$ be a bounded domain in Gaussian shrinking soliton $(\mathbb{R}^n,\langle,\rangle_{can}, \frac{1}{4}|x|^2)$ and $\lambda_i$ be the $i$-th eigenvalue of Problem~\eqref{problem1} for the drifted Laplacian operator on $\Omega$. 
Then, we have
\begin{equation*}
    \sum_{i=1}^k(\lambda_{k+1}-\lambda_i)^2 \leq \frac{4}{n}\sum_{i=1}^k(\lambda_{k+1}-\lambda_i)\Big(\lambda_i+\frac{n}{4}-\frac{1}{16}\inf_\Omega |x|^2\Big),
\end{equation*}
and
\begin{equation*}
    \sum_{i=1}^n(\lambda_{k+1}-\lambda_1) \leq 4\Big(\lambda_1+\frac{n}{4}-\frac{1}{16}\inf_\Omega|x|^2\Big).
\end{equation*}
In particular, for any positive real number $r_0>4n$ we consider the annular bounded domain $\Omega = \big\{ x \in \mathbb{R}^n; 4n < |x|^2 < r_0^2 \big\}$ so that $\inf_\Omega|x|^2=4n$, and then
\begin{equation}\label{eigenvalues-invariance}
    \sum_{i=1}^k(\lambda_{k+1}-\lambda_i)^2 \leq \frac{4}{n}\sum_{i=1}^k(\lambda_{k+1}-\lambda_i)\lambda_i, \quad \mbox{and} \quad \sum_{i=1}^n(\lambda_{k+1}-\lambda_1) \leq 4\lambda_1.
\end{equation}
\end{cor}
\begin{remark}\label{remark-1.2}
Notice that inequalities in \eqref{eigenvalues-invariance} do not depend on the constant $C_0$. They have the same behavior as the known eigenvalue estimates for the Laplacian case, see \cite{Yang} or \cite[Inequality~(1.7)]{ChengYang07}, and Chen and Cheng~\cite{ChenCheng08}, respectively.
\end{remark}

\section{Preliminaries}\label{preliminaries}

This section is short and serves to establish some basic notations and describe what is meant by properties of a $(1,1)$-tensor in a bounded domain $\Omega\subset M^n$ with smooth boundary $\partial\Omega$.

Throughout the paper, we are assuming the domains to be connected. Also, we are constantly using the identification of a $(0,2)$-tensor $T:\mathfrak{X}(M)\times\mathfrak{X}(M)\to~C^{\infty}(M)$ with its associated $(1,1)$-tensor $T:\mathfrak{X}(M)\to\mathfrak{X}(M)$ by the equation
\begin{equation*}
    \langle T(X), Y \rangle = T(X, Y).
\end{equation*}
In particular, the tensor $\langle , \rangle$ will be identified with the identity $I$ in $\mathfrak{X}(M)$. We observe that 
$\varepsilon \leq\langle T(X), X \rangle\leq\delta$, for any unit vector field $X$ on $\Omega$, implies
\begin{align}\label{T-property}
\varepsilon  \langle T(Y), Y\rangle \leq |T(Y)|^2 \leq \delta  \langle T(Y), Y\rangle \quad \mbox{for all} \quad Y \in \mathfrak{X}(M).
\end{align}
So, 
\begin{align}\label{T-norm}
    \varepsilon^2|\nabla \eta|^2 \leq |T(\nabla\eta)|^2 \leq \delta^2|\nabla\eta|^2.
\end{align}

For an $n$-dimensional complete Riemannian manifold $(M^n, \langle , \rangle)$  isometrically immersed in $\mathbb{R}^m$ we denote by $\alpha$ its second fundamental form and by ${\bf H}=\frac{1}{n}\tr (\alpha)$ its mean curvature vector. For a symmetric $(1, 1)$-tensor $T$ we have
\begin{equation*}
    {\bf H}_T=\frac{1}{n}\sum_{i,j=1}^nT(e_i, e_j)\alpha(e_i, e_j)=\frac{1}{n}\sum_{i=1}^n\alpha(T(e_i), e_i)=\frac{1}{n}\tr{(\alpha\circ T)}
\end{equation*}
where $\{e_1, e_2, \ldots, e_n\}$ is an orthonormal basis of $T_pM$ and ${\bf H}_T$ is called the generalized mean curvature vector at $p\in M$. We will use the following notations
\begin{align*}
    \tr(\nabla T):=\sum_{i=1}^n(\nabla T)(e_i, e_i), \quad |T|=\Big(\sum_{i=1}^n|T(e_i)|^2\Big)^{\frac{1}{2}}
\end{align*}
 and $\|\cdot\|_{L^2}$ for the canonical norm of a function in $L^2(\Omega,e^{-\eta}d\Omega)$. 
 
 Since $T$ is symmetric, notice that $\nabla_X T$ is also symmetric for each $X\in \mathfrak{X}(M)$, that is,
 \begin{equation}\label{symmetric-tensor}
     \langle (\nabla_X T)Y, Z\rangle = \langle Y, (\nabla_X T)Z \rangle, \quad \forall Y, Z \in \mathfrak{X}(M).
 \end{equation}
If $T$ is symmetric and divergence-free, then from \eqref{symmetric-tensor} we have
\begin{align*}
    0 = \dv T(X) = \sum_{i=1}^n \langle (\nabla_{e_i}T)(X), e_i \rangle = \sum_{i=1}^n \langle X, (\nabla_{e_i}T)(e_i)\rangle
    = \langle X, \sum_{i=1}^n (\nabla_{e_i}T)(e_i)\rangle
\end{align*}
for all $X \in \mathfrak{X}(M)$, and then $\tr(\nabla T)=\sum_{i=1}^n (\nabla_{e_i}T)(e_i) = 0$. 

Definition of $\eta$-divergence of $X$ (see Eq.~\eqref{eq1.1}) implies that
\begin{equation*}
    \dv_\eta(fX)=f\dv_\eta X + \langle \nabla f, X \rangle
\end{equation*}
and then
\begin{align}\label{property1}
    \mathscr{L}(f\ell) = f\mathscr{L}\ell + 2T(\nabla f, \nabla \ell) + \ell\mathscr{L}f
\end{align}
for all $f,\ell \in C^\infty(M)$. Besides, the divergence theorem is valid as follows:
\begin{equation}\label{property2}
   \int_\Omega\dv_\eta X dm = \int_{\partial\Omega} \langle X, \nu\rangle d\mu,
\end{equation}
where $dm = e^{-\eta}d\Omega$ is the weighted volume form on $\Omega$ and $d\mu = e^{-\eta}d\partial\Omega$ is the weighted area form on $\partial\Omega$ induced by the 
outward pointing unit normal vector field $\nu$ along $\partial \Omega$. In particular, by taking $X=T(\nabla f)$, we get 
\begin{equation*}
   \int_\Omega\mathscr{L}{f}dm = \int_{\partial\Omega}T(\nabla f, \nu) d\mu
\end{equation*}
and the integration by parts formula:
\begin{equation}\label{parts}
     \int_{\Omega}\ell\mathscr{L}{f}dm =-\int_\Omega T(\nabla\ell, \nabla f)dm + \int_{\partial\Omega}\ell T(\nabla f, \nu) d\mu
\end{equation}
for all $\ell, f \in C^\infty(M)$. Therefore, $\mathscr{L}$ is a formally self-adjoint operator in the Hilbert space $\mathcal{H}_0^1(\Omega,dm)$. Eigenspaces belonging to distinct eigenvalues are orthogonal in $L^2(\Omega,dm)$, which is the direct sum of all the eigenspaces. We refer
to the dimension of each eigenspace as the multiplicity of the eigenvalue. Thus the eigenvalue problem~\eqref{problem1} has a real and discrete spectrum
\begin{equation*}
    0 < \lambda_1 \leq \lambda_2 \leq \cdots \leq \lambda_k \leq \cdots\to\infty,
\end{equation*}
where each $\lambda_i$ is repeated according to its multiplicity. For eigenfunction $u_j$ corresponding to the eigenvalue $\lambda_j$, from Problem~\eqref{problem1} and integration by parts formula~\eqref{parts}, we have
\begin{equation}\label{lambda_i}
    \lambda_j\int_\Omega u_j^2dm = -\int_\Omega u_j\mathscr{L}u_jdm = \int_\Omega T(\nabla u_j, \nabla u_j)dm.
\end{equation}

To finish this section, we would like to refer here to the paper by the second author and Miranda~\cite[Section~2]{GomesMiranda} where is possible to find some geometric motivations to work with the operator $\mathscr{L}$ in the $(\eta,T)$-divergence form  and a Bochner-type formula for it on Riemannian manifolds. Besides, we highlight that the first author studied in~\cite{araujo2022inequalities} eigenvalues inequalities for $\mathscr{L}^2$ on complete Riemannian manifolds.

\section{Keystone technical lemmas}

In order to prove our results, we will need two technical lemmas. The first one has been proved by the second author and Miranda.
\begin{lem}[Gomes and Miranda\cite{GomesMiranda}]\label{lemma1}
Let $\Omega$ be a bounded domain in an $n$-dimensional complete Riemannian manifold $M$. Let $\lambda_i$ be the $i$-th eigenvalue of Problem~\eqref{problem1} and let $u_i$ be an $L^2(\Omega, dm)$-normalized real-valued eigenfunction corresponding to $\lambda_i$. Then, for any $f\in C^3(\Omega)\cap C^2(\partial \Omega)$, and $k$ integer, is valid
\begin{align*}
    \sum_{i=1}^k(\lambda_{k+1}-\lambda_i)^2 \int_{\Omega}T(\nabla f, \nabla f)u_i^2dm \leq 4\sum_{i=1}^k(\lambda_{k+1}-\lambda_i)\|T(\nabla f, \nabla u_i) + \frac{1}{2}u_i\mathscr{L}f \|^2_{L^2}.
\end{align*}
\end{lem}

From Lemma~\ref{lemma1} and similar discussions as in the proof of Proposition~2 in \cite{GomesMiranda} we obtain our keystone technical lemma.
\begin{lem}\label{lemma2}
Let $\Omega$ be a bounded domain in an $n$-dimensional complete Riemannian manifold $M$ isometrically immersed in $\mathbb{R}^m$, $\lambda_i$ be the $i$-th eigenvalue of Problem~\eqref{problem1} and $u_i$ be an $L^2(\Omega, dm)$-normalized real-valued eigenfunction corresponding to $\lambda_i$. Then is valid
\begin{align*}
   &\sum_{i=1}^k (\lambda_{k+1}-\lambda_i)^2 \int_\Omega \tr(T)u_i^2dm\\
      \leq& 4\sum_{i=1}^k(\lambda_{k+1}-\lambda_i)\Bigg\{\|T (\nabla u_i)\|^2_{L^2} +\frac{n^2}{4}\int_{\Omega}u_i^2|{\bf H}_T|^2dm+ \int_\Omega u_i T(\tr(\nabla T), \nabla u_i)dm\\
   &+\int_\Omega u_i^2\Big(\frac{1}{2}\dv (T^2(\nabla \eta)) - \frac{1}{4}|T(\nabla \eta)|^2\Big)dm\\
   &+\frac{1}{4}\int_\Omega u_i^2\langle\tr(\nabla T), \tr(\nabla T) - 2T(\nabla \eta)\rangle dm \Bigg\}.
\end{align*}

\end{lem}
\begin{proof} Let $x=(x_1, \ldots, x_m)$ be the position vector of the immersion of $M$ in $\mathbb{R}^m$. Taking $f=x_\ell$ in Lemma~\ref{lemma1} and summing over $\ell$ from 1 to $m$, we get
\begin{align}\label{eq4.17}
    &\sum_{i=1}^k(\lambda_{k+1}-\lambda_i)^2\sum_{\ell=1}^m\int_{\Omega}T(\nabla x_\ell, \nabla x_\ell) u_i^2dm\nonumber\\
    \leq& 4\sum_{i=1}^k(\lambda_{k+1}-\lambda_i)\int_\Omega\sum_{\ell=1}^m\Big|T(\nabla x_\ell, \nabla u_i) + \frac{1}{2}\mathscr{L} x_\ell u_i\Big|^2dm\nonumber\\
    =&4\sum_{i=1}^k(\lambda_{k+1}-\lambda_i)\int_{\Omega}\sum_{\ell=1}^m\Bigg[\frac{u_i^2}{4}(\mathscr{L} x_\ell)^2 + u_i \mathscr{L} x_\ell T(\nabla x_\ell, \nabla u_i)\\
    &+ |T(\nabla x_\ell, \nabla u_i)|^2 \Bigg]dm.\nonumber
\end{align}
Let $\{e_1, \ldots, e_m\}$ be a local orthonormal geodesic frame at $p\in M$ adapted to $M$. By a straightforward computation, similarly to Eq.~(3.17)-(3.24) in~\cite{GomesMiranda}, we obtain
\begin{align}\label{equation3.2}
      \quad \quad \sum_{\ell=1}^mT(\nabla x_\ell, \nabla x_\ell)=\sum_{\ell=1}^n\langle T(e_\ell), e_\ell \rangle = \tr(T),
\end{align}
\begin{align}\label{equation3.3}
      \sum_{\ell=1}^m|T(\nabla x_\ell, \nabla u_i)|^2= \sum_{\ell=1}^m | T(e_\ell, \nabla u_i)|^2 =|T(\nabla u_i)|^2,
\end{align}
\begin{align}
    \sum_{\ell=1}^m (\mathscr{L} x_\ell)^2&= |\tr{(\alpha\circ T)}|^2+|\tr(\nabla T) - T(\nabla \eta)|^2\nonumber\\
    &=n^2|{\bf H}_T|^2+\langle\tr(\nabla T), \tr(\nabla T) - 2T(\nabla \eta)\rangle+|T(\nabla \eta)|^2,
\end{align}
and
\begin{align}\label{equation3.5}
    &\sum_{\ell=1}^m\mathscr{L} x_\ell T(\nabla x_\ell, \nabla u_i)=T(\tr(\nabla T), \nabla u_i) - T(T(\nabla \eta), \nabla u_i).
\end{align}
Substituting \eqref{equation3.2}-\eqref{equation3.5} into \eqref{eq4.17} we get
\begin{align}\label{eq4.18}
   &\sum_{i=1}^k (\lambda_{k+1}-\lambda_i)^2 \int_\Omega \tr(T)u_i^2dm\nonumber\\
      \leq& 4\sum_{i=1}^k(\lambda_{k+1}-\lambda_i)\Bigg\{\|T (\nabla u_i)\|_{L^2} +\frac{n^2}{4}\int_{\Omega}u_i^2|{\bf H}_T|^2dm+ \int_\Omega u_i T(\tr(\nabla T), \nabla u_i)dm \nonumber\\
   &+\frac{1}{4}\int_\Omega u_i^2|T(\nabla \eta)|^2dm- \int_\Omega u_i  T(T(\nabla \eta), \nabla u_i)dm\\
   &+\frac{1}{4}\int_\Omega u_i^2\langle\tr(\nabla T), \tr(\nabla T) - 2T(\nabla \eta)\rangle dm \Bigg\}.\nonumber
\end{align}
Since $u_i|_{\partial \Omega}=0$ by divergence theorem \eqref{property2}, we have
\begin{align}\label{Eq-3.7}
    -&\int_\Omega u_i T(T(\nabla \eta),\nabla u_i) dm = -\frac{1}{2}\int_\Omega \langle T^2(\nabla \eta), \nabla u_i^2\rangle dm =\frac{1}{2}\int_\Omega u_i^2 \dv_\eta (T^2(\nabla \eta))dm.
\end{align}
Substituting the previous equality into Inequality~\eqref{eq4.18} and noticing that
\begin{align*}
    \dv_\eta (T^2(\nabla \eta)) = \dv (T^2(\nabla \eta)) - |T(\nabla \eta)|^2
\end{align*}
we complete the proof of the lemma.
\end{proof}

We now are in a position to prove the main theorems of this paper.

\section{Proof of theorems}
\subsection{Proof of Theorem~\ref{theorem1.1}}
\begin{proof}
The proof is a consequence of Lemma~\ref{lemma2}. We start by calculating
\begin{align*}
    &\frac{1}{4}\int_{\Omega}u_i^2\langle \tr(\nabla T), \tr(\nabla T) -2 T(\nabla \eta) \rangle dm\\
    =& \frac{1}{4} \int_{\Omega} u_i^2 | \tr(\nabla T)|^2dm - \frac{1}{2}\int_{\Omega}u_i^2\langle \tr(\nabla T), T(\nabla \eta) \rangle dm.
\end{align*}
Since $T_0=\sup_{\Omega}|\tr(\nabla T)|$ and $\eta_0=\sup_{\Omega}|\nabla \eta|$, we have
\begin{align*}
    \frac{1}{4} \int_{\Omega} u_i^2 | \tr(\nabla T)|^2dm \leq \frac{1}{4}T_0^2\int_\Omega u_i^2 dm =\frac{T_0^2}{4} ,
\end{align*}
and using \eqref{T-norm} we get
\begin{align*}
    -\frac{1}{2}\int_{\Omega}u_i^2\langle \tr(\nabla T), T(\nabla \eta) \rangle dm & \leq \frac{1}{2}\int_{\Omega}u_i^2|\tr(\nabla T)| |T(\nabla \eta)| dm \leq  \frac{\delta T_0\eta_0}{2}.
\end{align*}
Then,
\begin{align}\label{Equation-4.1}
    \frac{1}{4}\int_{\Omega}u_i^2\langle \tr(\nabla T), \tr(\nabla T) -2 T(\nabla \eta) \rangle dm \leq \frac{T_0^2}{4} +\frac{\delta T_0\eta_0}{2}.
\end{align}
Furthermore,
\begin{align}\label{eqq31}
    \int_\Omega u_i T(\tr(\nabla T), \nabla u_i) dm &\leq \Big(\int_\Omega u_i^2dm \Big)^{\frac{1}{2}} \Big(\int_\Omega \langle \tr(\nabla T), T(\nabla u_i)\rangle^2dm \Big)^{\frac{1}{2}} \nonumber\\
    & \leq T_0 \Big(\int_\Omega |T(\nabla u_i)|^2 dm\Big)^\frac{1}{2} = T_0\|T(\nabla u_i)\|_{L^2}
\end{align}
and
\begin{equation}\label{inequation-HT}
    \frac{n^2}{4}\int_\Omega u_i^2|{\bf H}_T|^2dm \leq \frac{n^2H_0^2}{4}\int_\Omega u_i^2dm = \frac{n^2H_0^2}{4},
\end{equation}
where $H_0= \sup_{\Omega}|{\bf H}_T|$. Since there exist positive real numbers $\varepsilon \langle X, X\rangle \leq \langle T(X), X\rangle \leq \delta \langle X, X \rangle$, for any vector field $X$ on $\Omega$, we have $n\varepsilon \leq \tr{(T)}$ and consequently
\begin{align}
    n\varepsilon = n\varepsilon \int_\Omega u_i^2dm \leq \int_\Omega \tr{(T)}u_i^2 dm.
\end{align}
Let us consider $C_0 = \sup_\Omega \Big\{\frac{1}{2}\dv (T^2(\nabla \eta)) - \frac{1}{4}|T(\nabla \eta)|^2\Big\}$ so that
\begin{align}\label{Equation-4.5}
    \int_\Omega u_i^2\Big(\frac{1}{2}\dv (T^2(\nabla \eta)) - \frac{1}{4}|T(\nabla \eta)|^2\Big)dm \leq C_0.
\end{align}
Substituting \eqref{Equation-4.1}-\eqref{Equation-4.5} into Lemma~\ref{lemma2}, we obtain
\begin{align}\label{equation5-3}
     n\varepsilon&\sum_{i=1}^k(\lambda_{k+1}-\lambda_i)^2 \nonumber\\
     \leq& 4\sum_{i=1}^k(\lambda_{k+1}-\lambda_i)\Big\{\|T(\nabla u_i)\|_{L^2}^2  + \frac{T_0^2}{4} + T_0\|T(\nabla u_i)\|_{L^2}\nonumber +\frac{\delta T_0\eta_0}{2}+ C_0 +\frac{n^2H_0^2}{4}\Big\}\nonumber\\
     =& 4 \sum_{i=1}^k(\lambda_{k+1}-\lambda_i)\Big\{\Big(\|T(\nabla u_i)\|_{L^2}  + \frac{1}{2}T_0\Big)^2 + \frac{n^2H_0^2+4C_0+2\delta T_0\eta_0}{4}\Big\}.
\end{align}
Moreover, from \eqref{T-property} and \eqref{lambda_i} we have
\begin{equation}\label{equation5-4}
    \|T(\nabla u_i)\|_{L^2}^2=\int_{\Omega} T(T(\nabla u_i), \nabla u_i) dm \leq \delta \int_{\Omega} T(\nabla u_i, \nabla u_i) dm = \delta\lambda_i.
\end{equation}
Therefore, from \eqref{equation5-3} and \eqref{equation5-4} we get
\begin{align*}
     \sum_{i=1}^k(\lambda_{k+1}-\lambda_i)^2 &\leq \frac{4}{n\varepsilon} \sum_{i=1}^k(\lambda_{k+1}-\lambda_i)\Big[\Big(\sqrt{\delta\lambda_i}  + \frac{1}{2}T_0 \Big)^2 + \frac{n^2H_0^2+4C_0+2\delta T_0\eta_0}{4}\Big]\\
     &= \frac{4\delta}{n\varepsilon} \sum_{i=1}^k(\lambda_{k+1}-\lambda_i)\Big[\Big(\sqrt{\lambda_i}  + \frac{T_0}{2\sqrt{\delta}} \Big)^2 + \frac{n^2H_0^2+4C_0+2\delta T_0\eta_0}{4\delta}\Big],
\end{align*}
which complete the proof of Theorem~\ref{theorem1.1}.
\end{proof}

\subsection{Proof of Theorem~\ref{theorem1.2}}
\begin{proof}
Let $x=(x_1, \ldots, x_m)$ be the position vector of the immersion of $M$ in $\mathbb{R}^m$. Let us consider the matrix $D=(d_{ij})_{m \times m}$ where
\begin{equation*}
    d_{ij}:= \int_\Omega x_iu_1 u_{j+1} dm.
\end{equation*}
From the orthogonalization of Gram and Schmidt, we know that there exists an upper triangle matrix $R=(r_{ij})_{m \times m}$ and an orthogonal matrix $S=(s_{ij})_{m \times m}$ such that $R=SD$, namely
\begin{equation*}
    r_{ij}=\sum_{k=1}^m s_{ik}d_{kj}= \sum_{k=1}^m s_{ik} \int_\Omega x_ku_1 u_{j+1} dm = \int_\Omega \Big( \sum_{k=1}^m s_{ik}x_k\Big)u_1 u_{j+1} dm = 0,
\end{equation*}
for $1 \leq j < i \leq m$. By setting $y_i=\sum_{k=1}^m s_{ik}x_k$, we have
\begin{equation*}
    \int_\Omega y_iu_1 u_{j+1} dm = 0   \quad \mbox{for} \quad 1 \leq j < i \leq m.
\end{equation*}
Let us denote $a_i=\int_\Omega y_i |u_1|^2dm$ and consider the real-valued functions $w_i$ given by
\begin{equation*}
    w_i = (y_i - a_i)u_1,
\end{equation*}
so that
\begin{equation*}
w_i|_{\partial \Omega}=0 \quad \mbox{and} \quad  \int_\Omega w_i u_{j+1} dm =0, \quad \mbox{for any} \quad j= 1, \ldots, i-1.
\end{equation*}
Then, from Rayleigh-Ritz inequality, we have for $1\leq i \leq m$
\begin{align}\label{Rayleigh-Ritz}
    \lambda_{i+1}\|w_i\|_{L^2}^2 \leq -\int_\Omega w_i \mathscr{L}w_i dm.
\end{align}
From definition of $w_i$ and using \eqref{property1} we get
\begin{align}\label{Equation4.8}
    -\int_\Omega w_i \mathscr{L}w_i dm &= -\int_\Omega w_i\Big[(y_i-a_i)\mathscr{L}u_1+u_1\mathscr{L}y_i+2T(\nabla y_i, \nabla u_1)\Big]dm\nonumber\\
    &=\lambda_1\|w_i\|_{L^2}^2 - \int_\Omega w_i (u_1 \mathscr{L}y_i + 2T(\nabla y_i, \nabla u_1))dm.
\end{align}
From \eqref{Rayleigh-Ritz} and \eqref{Equation4.8} we obtain
\begin{align}\label{Equation4.9}
   (\lambda_{i+1}- \lambda_1)\|w_i\|_{L^2}^2 \leq - \int_\Omega w_i (u_1 \mathscr{L}y_i + 2T(\nabla y_i, \nabla u_1))dm.
\end{align}
Using the Cauchy-Schwarz inequality, we have
\begin{align*}
   &(\lambda_{i+1}-\lambda_1)\Big(-\int_\Omega w_i (u_1 \mathscr{L}y_i + 2T(\nabla y_i, \nabla u_1))dm\Big)^2\nonumber\\
   &\leq (\lambda_{i+1}-\lambda_1) \|w_i\|_{L^2}^2 \|u_1 \mathscr{L}y_i + 2T(\nabla y_i, \nabla u_1)\|_{L^2}^2.
\end{align*}
Hence, by the previous inequality and \eqref{Equation4.9} we infer
\begin{align}\label{Equation4.10}
   (\lambda_{i+1}-\lambda_1)\Big(-\int_\Omega w_i &(u_1 \mathscr{L}y_i + 2T(\nabla y_i, \nabla u_1))dm\Big)\leq \|u_1 \mathscr{L}y_i + 2T(\nabla y_i, \nabla u_1)\|_{L^2}^2.
\end{align}
Integration by parts formula \eqref{parts} give us
\begin{align*}
    - \int_\Omega w_i (u_1 \mathscr{L}y_i + 2T(\nabla y_i, \nabla u_1))dm =  \int_\Omega |u_1|^2T(\nabla y_i, \nabla y_i) dm,
\end{align*}
and substituting the previous equality into \eqref{Equation4.10} we obtain
\begin{align}\label{equationn4.25}
    (\lambda_{i+1} - \lambda_1)\int_\Omega |u_1|^2T(\nabla y_i, \nabla y_i) dm \leq  4\Big\|\frac{1}{2}u_1 \mathscr{L}y_i + T(\nabla y_i, \nabla u_1)\Big\|_{L^2}^2.
\end{align}
Summing over $i$ from $1$ to $m$ in \eqref{equationn4.25}, we have
\begin{align}\label{equationn4.27}
\sum_{i=1}^m&(\lambda_{i+1} - \lambda_1)\int_\Omega |u_1|^2T(\nabla y_i, \nabla y_i) dm \leq 4\sum_{i=1}^m\Big\| \frac{1}{2}u_1 \mathscr{L}y_i + T(\nabla y_i, \nabla u_1) \Big\|_{L^2}^2.
\end{align}
Hence, from definition of $y_i$ and using \eqref{equation3.2}-\eqref{equation3.5} and the fact that $S$ is an orthogonal matrix, we obtain
\begin{align}\label{Equation-4.14}
       \sum_{\ell=1}^m|\nabla y_\ell|^2=n, \quad \sum_{\ell=1}^m|\nabla y_\ell u_i|^2=u_i^2,
\end{align}
\begin{align}
      \sum_{\ell=1}^m|T(\nabla y_\ell, \nabla u_i)|^2= \sum_{\ell=1}^m | T(e_\ell, \nabla u_i)|^2 =|T(\nabla u_i)|^2,
\end{align}
\begin{align}\label{Equation-4.16}
    \sum_{\ell=1}^m (\mathscr{L}(y_\ell))^2= n^2|{\bf H}_T|^2+|\tr(\nabla T) - T(\nabla \eta)|^2.
\end{align}
Since there exist positive real numbers $\varepsilon$ and $\delta$ such that $\varepsilon \langle X, X\rangle \leq \langle T(X), X\rangle \leq \delta \langle X, X \rangle$, for any vector field $X$ on $\Omega$, we have
\begin{align}\label{Equation4.13}
    \sum_{i=1}^m(\lambda_{i+1} - \lambda_1)T(\nabla y_i, \nabla y_i) \geq \varepsilon \sum_{i=1}^m(\lambda_{i+1} - \lambda_1)|\nabla y_i|^2,
\end{align}
and
\begin{align*}
    \sum_{i=1}^m(\lambda_{i+1} - \lambda_1)|\nabla y_i|^2&\geq \sum_{i=1}^n(\lambda_{i+1} - \lambda_1)|\nabla y_i|^2+(\lambda_{n+1} - \lambda_1)\sum_{\gamma=n+1}^m|\nabla y_\gamma|^2\\
    &=\sum_{i=1}^n(\lambda_{i+1} - \lambda_1)|\nabla y_i|^2+(\lambda_{n+1}-\lambda_1)\sum_{i=1}^n(1-|\nabla y_i|^2)\\
    &\geq \sum_{i=1}^n(\lambda_{i+1} - \lambda_1)|\nabla y_i|^2+\sum_{i=1}^n(\lambda_{i+1}-\lambda_1)(1-|\nabla y_i|^2)\\
    &=\sum_{i=1}^n(\lambda_{i+1}-\lambda_1).
\end{align*}
So, from the previous inequality and \eqref{Equation4.13}
\begin{align}\label{Equation4.14}
    \sum_{i=1}^m(\lambda_{i+1} - \lambda_1)T(\nabla y_i, \nabla y_i) \geq \varepsilon \sum_{i=1}^n(\lambda_{i+1}-\lambda_1).
\end{align}
Analogous to the proof of Theorem~\ref{theorem1.1}, using identities~\eqref{Equation-4.14}-\eqref{Equation-4.16},  we obtain
\begin{equation}\label{Equation4.15}
   0 < \sum_{i=1}^m\Big\| \frac{1}{2}u_1 \mathscr{L}y_i + T(\nabla y_i, \nabla u_1) \Big\|_{L^2}^2 \leq \delta\Big[\Big(\sqrt{\lambda_1}+\frac{T_0}{2\sqrt{\delta}}\Big)^2 + \frac{n^2H_0+4C_0+2\delta T_0\eta_0}{4\delta}\Big].
\end{equation}
Thus, from \eqref{equationn4.27}, \eqref{Equation4.14} and \eqref{Equation4.15} we conclude that
\begin{align*}
    \sum_{i=1}^n(\lambda_{i+1}-\lambda_1)\leq\frac{4\delta}{\varepsilon}\Big[\Big(\sqrt{\lambda_1}+\frac{T_0}{2\sqrt{\delta}}\Big)^2 + \frac{n^2H_0+4C_0+2\delta T_0\eta_0}{4\delta}\Big],
\end{align*}
and so we complete the proof of Theorem~\ref{theorem1.2}.
\end{proof}
\subsection{Proof of Corollary~\ref{cor-yang-type-ineq}}
\begin{proof}
We begin by noting that \ref{T-parallel} is equivalent to
\begin{equation*}
    k\varsigma_{k+1}^2 -\Big(2 + \frac{4\delta}{n\varepsilon}\Big)\varsigma_{k+1} \sum_{i=1}^k\varsigma_i + \Big(1 + \frac{4\delta}{n\varepsilon}\Big)\sum_{i=1}^k\varsigma_i^2\leq 0,
\end{equation*}
which has discriminant non-negative. So, we can follows the steps in the proof of \cite[Theorem~3]{GomesMiranda} to solve it and to obtain the inequalities in Corollary~\ref{cor-yang-type-ineq}.
\end{proof}

\subsection{Proof of Corollary~\ref{cor-conjectura}}
\begin{proof}
The proof follows from the same steps as in the corresponding result of~\cite{GomesMiranda}. For the drifted Laplacian case, Inequality~\eqref{T-parallel} becomes
\begin{align}\label{T-parallel-2}
     \sum_{i=1}^k(\varsigma_{k+1}-\varsigma_i)^2 \leq \frac{4}{n}\sum_{i=1}^k(\varsigma_{k+1}-\varsigma_i)\varsigma_i.
\end{align}
where $\varsigma_i=\lambda_i+\frac{n^2H_0^2+4C_0}{4}$ with $C_0 = \sup_\Omega \{\frac{1}{2}\Delta \eta - \frac{1}{4}|\nabla \eta|^2\}$ and $H_0=\sup_\Omega |{\bf H}|$. From \eqref{T-parallel-2} and the recursion formula of Cheng and Yang~\cite{ChengYang07} we get
\begin{equation}\label{Ineq-4.22}
    \frac{F_{k+l}}{(k+l)^{\frac{4}{n}}} \leq \frac{F_k}{k^{\frac{4}{n}}},
\end{equation}
for any positive integer $l$, where $F_k:=(1+\frac{2}{n})\Big(\frac{1}{k}\sum_{i=1}^k\varsigma_i\Big)^2 - \frac{1}{k}\sum_{i=1}^k\varsigma_i^2$. 
Furthermore, we see that
\begin{equation}
    \frac{F_k}{k^{\frac{4}{n}}} \leq \frac{\Big(\frac{1}{k}\sum_{i=1}^k\varsigma_i\Big)^2}{k^{\frac{4}{n}}}\,\, \mbox{and}\,\, \frac{F_{k+l}}{(k+l)^{\frac{4}{n}}} = \Big(1+\frac{2}{n}\Big)\frac{\Big(\frac{1}{k+l}\sum_{i=1}^k\varsigma_i\Big)^2}{(k+l)^{\frac{4}{n}}} - \frac{\frac{1}{k+l}\sum_{i=1}^k\varsigma_i^2}{(k+l)^{\frac{4}{n}}}.
\end{equation}
Using the asymptotic Weyl's formula we obtain
\begin{align}\label{Ineq-4.24}
   \lim_{k\to \infty} \frac{\frac{1}{k}\sum_{i=1}^k\varsigma_i}{k^{\frac{2}{n}}} = \frac{n}{n+2} \frac{4 \pi^2}{(\omega_n \vol \Omega)^\frac{2}{n}} \,\, \mbox{and} \,\, \lim_{k\to \infty} \frac{\frac{1}{k}\sum_{i=1}^k\varsigma_i^2}{k^{\frac{4}{n}}} = \frac{n}{n+4} \frac{16 \pi^4}{(\omega_n \vol \Omega)^\frac{4}{n}}.
\end{align}
From Inequalities~\eqref{Ineq-4.22}-\eqref{Ineq-4.24}, following the last part of the proof of \cite[Theorem~2]{GomesMiranda}, we complete the proof.
\end{proof}

\begin{remark}\label{Impovement-CPolya}
 Corollary~\ref{cor-conjectura} is an improvement of Theorem~2 in \cite{GomesMiranda}, since $4C_0=\sup_\Omega\Big\{2\Delta \eta-|\nabla \eta|^2\Big\}\leq 2\Bar{\eta_0}+\eta_0^2$, where $\Bar{\eta_0}=\sup_\Omega |\Delta_\eta \eta|$ and $\eta_0 = \sup_\Omega |\nabla \eta|$.  Likewise, our Corollary~\ref{cor-yang-type-ineq} improves  Theorem~3 in~\cite{GomesMiranda}. 
\end{remark}

\section{Concluding remarks}\label{concludingremarks}
In this section, we obtain another universal estimate for Problem~\eqref{problem1} that generalizes a result by Xia and Xu~\cite{XiaXu} for the drifted Laplacian operator. In a more general setting, we have the following result.
\begin{theorem}\label{theorem-5.1}
Let $\Omega$ be a bounded domain in an $n$-dimensional complete Riemannian manifold $M^n$ isometrically immersed in $\mathbb{R}^m$, and $\lambda_i$ be the $i$-th eigenvalue of Problem~\eqref{problem1}. Then, we have
\begin{equation*}
\sum_{i=1}^k(\lambda_{k+1}-\lambda_i)^2 \leq \frac{4\delta}{n\varepsilon}\sum_{i=1}^k(\lambda_{k+1}-\lambda_i)\Big[\lambda_i + \Big(\frac{T_0}{\sqrt{\delta}}+ \eta_0\sqrt{\delta}\Big)\sqrt{\lambda_i} + \frac{n^2H_0^2+(T_0 + \delta\eta_0)^2}{4\delta} \Big]
\end{equation*}
where $T_0=\sup_\Omega|\tr(\nabla T)|$, $\eta_0=\sup_\Omega|\nabla \eta|$ and $H_0=\sup_\Omega |{\bf H}_T|$.
\end{theorem}
\begin{proof}
The proof is a slight modification from the proof of Theorem~\ref{theorem1.1}. For this is enough to notice that
\begin{align*}
\int_\Omega u_i^2\Big(\frac{1}{2}\dv(T^2(\nabla \eta))-\frac{1}{4}|T(\nabla \eta)|^2\Big)dm =& \int_\Omega u_i^2\Big(\frac{1}{2}\dv_\eta (T^2(\nabla \eta)) +\frac{1}{4}|T(\nabla \eta)|^2 \Big)dm\\
=& - \int_\Omega u_i\langle T(\nabla \eta), T(\nabla u_i)\rangle dm\\
&+ \frac{1}{4}\int_\Omega u_i^2|T(\nabla \eta)|^2dm \\
\leq& \Big(\int_\Omega |T(\nabla \eta)|^2|T(\nabla u_i)|^2dm\Big)^{\frac{1}{2}} + \frac{\delta^2\eta_0^2}{4}\\
\leq& \delta \eta_0 \|T(\nabla u_i)\| + \frac{\delta^2\eta_0^2}{4} \leq \delta^{\frac{3}{2}} \eta_0 \sqrt{\lambda}  + \frac{\delta^2\eta_0^2}{4}
\end{align*}
where in the above calculation we have used \eqref{T-norm}, \eqref{Eq-3.7} and \eqref{equation5-4}.
\end{proof}
The next result is immediate from Theorem~\ref{theorem-5.1}.
\begin{cor}\label{cor-5.1}
Let $\Omega$ be a bounded domain in an $n$-dimensional complete Riemannian manifold $M^n$ isometrically immersed in $\mathbb{R}^m$, and $\lambda_i$ be the $i$-th eigenvalue of Problem~\ref{problem1} for the drifted Cheng-Yau operator. Then, we have
\begin{equation*}
    \sum_{i=1}^k(\lambda_{k+1}-\lambda_i)^2 \leq \frac{4\delta}{n\varepsilon}\sum_{i=1}^k(\lambda_{k+1}-\lambda_i)\left(\lambda_i + \eta_0\sqrt{\delta}\sqrt{\lambda_i} + \frac{n^2H_0^2+\delta^2\eta_0^2}{4\delta}\right),
\end{equation*}
where $\eta_0=\sup_\Omega|\nabla \eta|$ and $H_0=\sup_\Omega |{\bf H}_T|$.
\end{cor}
\begin{remark}
We highlight that Theorem~\ref{theorem-5.1} generalize \cite[Theorem~1.2, (i)]{XiaXu}. This fact is more evident by taking $T=I$ in Corollary~\ref{cor-5.1} so that $\varepsilon = \delta = 1$, and then the previous inequality becomes the inequality in \cite[Theorem~1.2, (i)]{XiaXu}.
\end{remark}

\section*{Acknowledgements}
The first author has been partially supported by Coordenação de Aperfeiçoamento de Pessoal de Nível Superior (CAPES) in conjunction with Fundação Rondônia de Amparo ao Desenvolvimento das Ações Científicas e Tecnológicas e à Pesquisa do Estado de Rondônia (FAPERO). The second author has been partially supported by Conselho Nacional de Desenvolvimento Científico e Tecnológico (CNPq), of the Ministry of Science, Technology and Innovation of Brazil. Moreover, would like to express our sincere thanks to the anonymous referee for his/her careful reading and useful comments which helped us improve our paper.

\end{document}